\documentclass[a4paper, reqno, 12pt]{amsart}

\usepackage[usenames,dvipsnames]{color}
\usepackage{amsthm,amsfonts,amssymb,amsmath,amsxtra}
\usepackage[all]{xy}
\SelectTips{cm}{}
\usepackage{xr-hyper}
\usepackage[colorlinks=
   citecolor=Black,
   linkcolor=Red,
   urlcolor=Blue]{hyperref}
\usepackage{verbatim}
\usepackage{enumerate}

\usepackage[margin=1.25in]{geometry}
\usepackage{mathrsfs}

\RequirePackage{xspace}
\RequirePackage{etoolbox}
\RequirePackage{varwidth}
\RequirePackage{enumitem}
\RequirePackage{tensor}
\RequirePackage{mathtools}
\RequirePackage{longtable}
\RequirePackage{multirow}

\def\ge{\geqslant}
\def\le{\leqslant}
\def\a{\alpha}
\def\b{\beta}

\def\G{\Gamma}
\def\d{\delta}

\def\o{\omega}

\def\s{\sigma}
\def\t{\tau}
\def\th{\theta}
\def\k{\kappa}
\def\l{\lambda}

\def\i{^{-1}}

\def\<{\langle}
\def\>{\rangle}

\newcommand{\bE}{\mathbf E}

\newcommand{\bK}{\mathbf K}

\def\brI{\breve I}
\def\brK{\breve K}

\newcommand{\BA}{\ensuremath{\mathbb {A}}\xspace}

\newcommand{\BF}{\ensuremath{\mathbb {F}}\xspace}
\newcommand{{\BG}}{\ensuremath{\mathbb {G}}\xspace}

\newcommand{\BJ}{\ensuremath{\mathbb {J}}\xspace}
\newcommand{{\BK}}{\ensuremath{\mathbb {K}}\xspace}

\newcommand{\BN}{\ensuremath{\mathbb {N}}\xspace}

\newcommand{\BQ}{\ensuremath{\mathbb {Q}}\xspace}
\newcommand{\BR}{\ensuremath{\mathbb {R}}\xspace}

\newcommand{\BZ}{\ensuremath{\mathbb {Z}}\xspace}

\newcommand{\CO}{\ensuremath{\mathcal {O}}\xspace}

\newcommand{\Ad}{{\mathrm{Ad}}}
\newcommand{\ad}{{\mathrm{ad}}}

\DeclareMathOperator{\Adm}{Adm}

\DeclareMathOperator{\Gal}{Gal}

\DeclareMathOperator{\ord}{ord}

\DeclareMathOperator{\Spec}{Spec}

\newcommand{\ov}{\overline}

\def\tW{\tilde W}

\newtheorem{theorem}{Theorem}
\newtheorem{proposition}[theorem]{Proposition}

\newtheorem{corollary}[theorem]{Corollary}

\theoremstyle{definition}

\newtheorem{example}[theorem]{Example}

\newtheorem{remark}[theorem]{Remark}

\numberwithin{equation}{section}
\numberwithin{theorem}{section}

\setitemize[0]{leftmargin=*,itemsep=\the\smallskipamount}
\setenumerate[0]{leftmargin=*,itemsep=\the\smallskipamount}

\renewcommand{\to}{%
   \ifbool{@display}{\longrightarrow}{\rightarrow}%
   }
\let\shortmapsto\mapsto
\renewcommand{\mapsto}{%
   \ifbool{@display}{\longmapsto}{\shortmapsto}%
   }
\newlength{\olen}
\newlength{\ulen}
\newlength{\xlen}
\newcommand{\xra}[2][]{%
   \ifbool{@display}%
      {\settowidth{\olen}{$\overset{#2}{\longrightarrow}$}%
       \settowidth{\ulen}{$\underset{#1}{\longrightarrow}$}%
       \settowidth{\xlen}{$\xrightarrow[#1]{#2}$}%
       \ifdimgreater{\olen}{\xlen}%
          {\underset{#1}{\overset{#2}{\longrightarrow}}}%
          {\ifdimgreater{\ulen}{\xlen}%
             {\underset{#1}{\overset{#2}{\longrightarrow}}}
             {\xrightarrow[#1]{#2}}}}%
      {\xrightarrow[#1]{#2}}
   }
\makeatother
\newcommand{\xyra}[2][]{%
   \settowidth{\xlen}{$\xrightarrow[#1]{#2}$}%
   \ifbool{@display}%
      {\settowidth{\olen}{$\overset{#2}{\longrightarrow}$}%
       \settowidth{\ulen}{$\underset{#1}{\longrightarrow}$}%
       \ifdimgreater{\olen}{\xlen}%
          {\mathrel{\xymatrix@M=.12ex@C=3.2ex{\ar[r]^-{#2}_-{#1} &}}}%
          {\ifdimgreater{\ulen}{\xlen}%
             {\mathrel{\xymatrix@M=.12ex@C=3.2ex{\ar[r]^-{#2}_-{#1} &}}}
             {\mathrel{\xymatrix@M=.12ex@C=\the\xlen{\ar[r]^-{#2}_-{#1} &}}}}}%
      {\mathrel{\xymatrix@M=.12ex@C=\the\xlen{\ar[r]^-{#2}_-{#1} &}}}%
   }
\makeatletter
\newcommand{\xla}[2][]{%
   \ifbool{@display}%
      {\settowidth{\olen}{$\overset{#2}{\longleftarrow}$}%
       \settowidth{\ulen}{$\underset{#1}{\longleftarrow}$}%
       \settowidth{\xlen}{$\xleftarrow[#1]{#2}$}%
       \ifdimgreater{\olen}{\xlen}%
          {\underset{#1}{\overset{#2}{\longleftarrow}}}%
          {\ifdimgreater{\ulen}{\xlen}%
             {\underset{#1}{\overset{#2}{\longleftarrow}}}
             {\xleftarrow[#1]{#2}}}}%
      {\xleftarrow[#1]{#2}}
   }
\newcommand{\isoarrow}{%
   \ifbool{@display}{\overset{\sim}{\longrightarrow}}{\xrightarrow\sim}%
   }

\begin{document}
\author{Xuhua He}
\address{Department of Mathematics, University of Maryland, College Park, MD 20742 and Institute for Advanced Study, Princeton, NJ 08540}
\email{xuhuahe@math.umd.edu}
\thanks{X. H. was partially supported by NSF DMS-1463852 and DMS-1128155 (from IAS). S. N. is supported in part by QYZDB-SSW-SYS007 and NSFC grant (No. 11501547, No. 11621061 and No. 11688101).}
\author{Sian Nie}
\address{Institute of Mathematics, Academy of Mathematics and Systems Science, Chinese Academy of Sciences, 100190, Beijing, China}
\email{niesian@amss.ac.cn}

\title[]{On the $\mu$-ordinary locus of a Shimura variety}
\begin{abstract}
In this paper, we study the $\mu$-ordinary locus of a Shimura variety with parahoric level structure. Under the axioms in \cite{HR}, we show that $\mu$-ordinary locus is a union of certain maximal Ekedahl-Kottwitz-Oort-Rapoport strata introduced in \cite{HR} and we give criteria on the density of the $\mu$-ordinary locus.
\end{abstract}

\maketitle

\section*{Introduction}

\subsection{} Let $\BA_{g, N}$ be the moduli space of principally polarized abelian varieties in characteristic $p$ with fixed dimension $g$ and level-$N$-structure. It is a classical result that the ordinary locus in $\BA_{g, N}$ is open and dense. This result is first proved by Koblitz \cite{Kob} by studying the $p$-rank stratification of $\BA_{g, N}$ via a deformation-theoretic argument. A different proof is obtained by Norman and Oort \cite{NO} by constructing the deformation using Cartier theory which raises the $p$-rank.

It is natural to study the similar problem for arbitrary Shimura varieties. However, the naive generalization does not work as the ordinary locus may be empty. In 1996, Rapoport \cite{Ra} formulated the $\mu$-ordinariness and conjectured the density of the $\mu$-ordinary locus in the hyperspecial case. For Shimura varieties of PEL-type,  this conjecture is proved by Wedhorn \cite{We} using a  deformation-theoretic argument. The result was reproved later by Moonen \cite{Mo} by showing that the $\mu$-ordinary locus coincides with the unique maximal (and open) Ekedahl-Oort stratum. Recently, Wortmann \cite{Wo} extended this result to Shimura varieties of Hodge type with hyperspecial level structure. 


For non-special level structure, it is a difficult problem to construct the deformation explicitly. In fact, the $\mu$-ordinary locus in a Shimura variety with Iwahori level structure is not dense in general, e.g. in the Hilbert-Blumenthal case \cite{St}. In \cite{Ha}, Hartwig gives a necessary and sufficient condition on the density of the ordinary locus in the symplectic PEL Shimura variety with Iwahori level structure.

\subsection{} The main purpose of this paper is to study the $\mu$-ordinary loci $Sh_K^{\mu-\ord}$ of Shimura varieties $Sh_K$ with arbitrary parahoric level structure.

One major novelty of our approach comes from the axioms in \cite{HR}, which provide a systematic way to study various stratifications of Shimura variety. The first step is to relate the $\mu$-ordinary locus with another stratification of $Sh_K$. For Iwahori level structure, we use the Kottwitz-Rapoport stratification. For hyperspecial level structure, we need to use  the Ekedahl-Oort stratification instead (see \cite{Mo}), as the KR stratification (which consists of a single stratum in this case) is insufficient. In the general case, we use the Ekedahl-Kottwitz-Oort-Rapoport stratification $Sh_K=\sqcup_x EKOR_{K, x}$ introduced in \cite{HR} which interpolates between the Kottwitz-Rapoport stratification in the case of Iwahori level structure and the Ekedahl-Oort stratification in the case of hyperspecial level structure.

We prove in Theorem \ref{union} that

\begin{theorem}
Under the axioms in \cite{HR}, the $\mu$-ordinary locus $Sh_K^{\mu-\ord}$ is a union of certain maximal Ekedahl-Kottwitz-Oort-Rapoport strata.
\end{theorem}

This result generalizes \cite{Mo} for Shimura varieties of unramified type with hyperspecial level structure. Compared with \cite{Mo}, a new phenomenon arising in our general case is that there exist more than one maximal Ekedahl-Kottwitz-Oort-Rapoport strata in general. It turns out that this is the only possible obstruction for the density of $\mu$-ordinary locus. We prove in Theorem \ref{density} the following equivalent criteria on the density of $\mu$-ordinary locus.

\begin{theorem}
Under the axioms in \cite{HR}, the following two conditions are equivalent:

(1) The $\mu$-ordinary locus $Sh_K^{\mu-\ord}$ is dense;

(2) The $\mu$-ordinary locus $Sh_K^{\mu-\ord}$ is the union of all maximal Ekedahl-Kottwitz-Oort-Rapoport strata of $Sh_K$.

If moreover, $G$ is quasi-split, then the above conditions are also equivalent to the following equality on the relative Weyl group $W_0$: $$p(W_K) W_0^\s W_\mu=W_0.$$
\end{theorem}

Here $p(W_K)$ is the image of the Weyl group $W_K$ of the parahoric subgroup $K$ in $W_0$; $W_0^\s$ is the group of $\s$-fixed elements in $W_0$; $W_\mu \subset W_0$ is the stabilizer of $\mu$.

As a consequence, if the group is residually split or $K$ is a maximal special parahoric subgroup, then the $\mu$-ordinary locus is dense.

\subsection{} We would like to draw attention that the main results above are subject to the axioms in~\cite{HR}. The nonemptiness of various characteristic subsets (e.g. the Newton stratification, the Ekedahl-Oort stratification and the Kottwitz-Rapoport stratification) in the reduction of a general Shimura variety has been an open problem in arithmetic geometry for many years. It is premature at present to expect unconditional results on the inclusion and closure relations (including the density problem) between these characteristic subsets for a general Shimura variety.

We stress (if necessary at all!) that our aim is quite modest. One purpose is to provide some necessary group-theoretic tools that could possibly be used to achieve further progress on the density problem of the $\mu$-ordinary locus of a Shimura variety with arbitrary parahoric level structure, while there are some difficulty to carry through the classical approach using deformation theory in this general situation. The other purpose is to provide an example of how the axiom set in \cite{HR} may be used in the theory of Shimura varieties: to simplify some arguments, to establish results in new cases and to obtain new results.

\subsection{} Let us come to the axioms of \cite{HR}. These axioms are by no means trivial, nor could be checked by a routine procedure. The situation is rather the opposite. First, the axiom set depends on the existence of integral models (e.g. the work of Rapoport-Zink, Kisin, and Kisin-Pappas). Even after the integral model is constructed, there still remains highly nontrivial work to verify that these integral models satisfy the axioms. We refer to \cite[\S 0.3]{HZ} for some detailed explanation.

The current status of the validity of these axioms is as follows. For PEL type Shimura varieties associated to unramified groups of type $A$ and $C$ and to odd ramified unitary groups (note that the reductive group $G$ in those cases are quasi-split), the axioms are verified by the first author and R. Zhou in \cite{HZ}. For Shimura varieties of Hodge type, there is an ongoing work by Zhou \cite{Zhou}, where most of the axioms are verified.

Upon the above progress, we have the following definitive (unconditional) result on the $\mu$-ordinary locus.

\begin{theorem}
For PEL type Shimura varieties associated to unramified groups of type $A$ and $C$ and to odd ramified unitary groups, we have the following results:

\begin{enumerate}
\item The $\mu$-ordinary locus $Sh_K^{\mu-\ord}$ is a union of certain maximal Ekedahl-Kottwitz-Oort-Rapoport strata.

\item The following conditions are equivalent:
\begin{enumerate}
\item The $\mu$-ordinary locus $Sh_K^{\mu-\ord}$ is dense;

\item The $\mu$-ordinary locus $Sh_K^{\mu-\ord}$ is the union of all maximal Ekedahl-Kottwitz-Oort-Rapoport strata of $Sh_K$.

\item The equality $p(W_K) W_0^\s W_\mu=W_0$ holds.
\end{enumerate}

\end{enumerate}
\end{theorem}


\subsection{} Finally, let us have a short discussion on the case where the $\mu$-ordinary locus is empty.

The natural range of the index set of the Newton strata is the set $B(G, \{\mu\})$ of neutral acceptable elements. By definition, all the elements in $B(G, \{\mu\})$ have Newton points less than or equal to $\mu^\diamond$, the Galois-average of a dominant representative in $\{\mu\}$. The $\mu$-ordinary locus $Sh^{\mu-\ord}_K$, if nonempty, is the unique maximal Newton stratum.

In certain cases, all the elements in $B(G, \{\mu\})$ are strictly less than $\mu^\diamond$. So the $\mu$-ordinary locus $Sh_K^{\mu-\ord}$ is empty. However, there is still a unique maximal element in $B(G, \{\mu\})$. Several rather different phenomena may occur: the maximal Newton stratum intersects some non-maximal EKOR stratum, and the maximal Newton stratum may not be a union of EKOR strata. We provide some examples in section \ref{non-qs}.

We conjecture that the maximal Newton stratum is dense if and only if it intersects every maximal EKOR stratum. In Proposition \ref{gp-den}, we prove the group-theoretic analogy of this conjecture.

\subsection{Acknowledgments}
We thank Ulrich G\"ortz and Michael Rapoport for valuable discussions and comments, and for drawing our attention to and for explaining the results of Hartwig \cite{Ha}.

\section{Preliminary}

\subsection{} We follow the set-up in \cite{HR}.

Let $({\bf G}, \{h\})$ be a Shimura datum and let ${\bf K}=K^p K$  be an open compact subgroup of ${\mathbf G}(\BA_f)$, where $\bf K^p \subset {\mathbf G}(\BA^p_f)$ and $K=K_p$ is a parahoric subgroup of ${\mathbf G}(\BQ_p)$. Let $G={\mathbf G} \otimes_\BQ \BQ_p$ and let $\{\mu\}$ be the conjugacy class of cocharacters of $G$ corresponding to $\{h\}$.

Let ${\rm Sh}_{\mathbf K}={\rm Sh}({\mathbf G}, \{h\})_{\mathbf K}$ be the corresponding Shimura variety. It is a quasi-projective variety defined over the Shimura field $\bE$. Let $O_E$ be the ring of integers of the completion $E$ of $\bE$ at a place ${\bf p}$ above the fixed prime number $p$. Let ${Sh}_K={\mathbf S}_{\bK}\times_{\Spec O_{\bE}}\Spec  \kappa_E$ be the special fiber.

\smallskip

{\bf Unless otherwise stated, in the rest of the paper, we assume the existence of an integral model ${\mathbf S_{\bK}}$ over $O_E$ and \cite[Axiom 1--5]{HR} holds.}

\subsection{} Let $\breve \BQ_p$ be the completion of the maximal unramified extension of $\BQ_p$ in a fixed algebraic closure $\ov\BQ_p$, with ring of integers $O_{\breve \BQ_p}$. We denote by $\sigma$ its Frobenius automorphism of $\breve\BQ_p$ over $\BQ_p$. Let $\Gamma=\Gal(\ov\BQ_p/\BQ_p)$ be the absolute Galois group and $\Gamma_0=\Gal(\ov\BQ_p/\BQ_p^{un})$ be the inertia subgroup.

Let $I$ be an Iwahori subgroup of $G$ that is contained in $K$. Let $\breve K$ be the parahoric subgroup of $G(\breve \BQ_p)$ associated to $K$ and $\brK_1$ the pro-unipotent radical of $\brK$. We fix a maximal torus $T$ which after extension of scalars is contained in a Borel subgroup of $G\otimes_{\BQ_p}\breve\BQ_p$, and such that $\breve I$ is the Iwahori subgroup fixing an alcove in the apartment attached to the split part of $T$.

Let $(X_*(T)_{\Gamma_0, \BQ})^+$ be the intersection of  $X_*(T)_{\Gamma_0}\otimes \BQ=X_*(T)^{\Gamma_0}\otimes \BQ$ with the set $X_*(T)_\BQ^+$ of dominant elements in $X_*(T)_\BQ$. Let $\pi_1(G)_\G$ be the set of $\G$-coinvariants of $\pi_1(G)$. The $\sigma$-conjugacy classes of $G(\breve\BQ_p)$ are classified by Kottwitz in \cite{Ko1} and \cite{Ko2} via the Newton map and the Kottwitz map (see \cite[(2.5) \& (2.6)]{HR})
$$(\nu, \k): B(G) \hookrightarrow \big((X_*(T)_{\Gamma_0, \BQ})^+\big)^{\langle\sigma\rangle} \times \pi_1(G)_\G.$$

The action of $\s$ on $X_*(T)_{\Gamma_0}\otimes \BQ$ induces a linear action $\s_0$ on $(X_*(T)_{\Gamma_0, \BQ})^+$ (the $L$-action). Let $B(G, \{\mu\})$ be the set of {\it neutral acceptable elements} with respect to $\{\mu\}$ in $B(G)$ defined by $$ B(G, \{\mu\})=\{ [b]\in B(G)\mid \kappa([b])=\mu^\natural, \nu([b])\le \mu^\diamond \}, $$ where $\mu^\natural$ denotes the common image of the elements of $\{\mu\}$ in $\pi_1(G)_\Gamma$, and $\mu^\diamond$ denotes the $\s_0$-average of a dominant representative $\mu$ of the image of an element  of $\{\mu\}$ in $X_*(T)_{\Gamma_0, \BQ}$.

By \cite[Axiom 3 \& Theorem 5.4]{HR}, there is a surjective map $$\delta_K: Sh_K\to B(G, \{\mu\}).
$$ For each $[b]\in B(G, \{\mu\})$, the fiber of $\delta_K$ over $[b]$ is the set of $\bar\kappa_E$-rational points of the {\it Newton stratum} $\mathit S_{K,[b]}$ of $Sh_K$ attached to $[b]$.\

\subsection{} We recall the Ekedahl-Kottwitz-Oort-Rapoport (in short, EKOR) stratification introduced in \cite[\S 6]{HR}. We first introduce some more notations.

Let $N$ be the normalizer of $T$. Then the {\it relative Weyl group} is defined by $W_0=N(\breve\BQ_p)/T(\breve\BQ_p)$ and the {\it Iwahori-Weyl group} is defined by $\tilde W=N(\breve\BQ_p)/(T(\breve\BQ_p)\cap \breve I)$. We fix a special vertex in the base alcove. Then $\tilde W$ is a split extension of $W_0$ by the abelian subgroup $X_*(T)_{\Gamma_0}$, with its natural $W_0$-action.

Note that the Frobenius morphism preserves the Iwahori subgroup $\brI$. Thus it induces a length-preserving automorphism on $\tW$, which we still denote by $\s$.

The {\it $\{\mu\}$-admissible set} is defined by
\[
\Adm(\{\mu\})=\{w \in \tilde W; w \le t^{x(\underline \mu)} \text{ for some }x \in W_0\} .
\]
Here $\underline \mu$ is the image in $X_*(T)_{\Gamma_0}$ of the dominant representative $\mu$ of $\{\mu\}$ and $\le$ is the Bruhat order on $\tilde W$.

Recall that $K$ is a parahoric subgroup containing $I$. Let
$$W_K=\tilde W\cap \breve K=\big(N(\breve\BQ_p)\cap \breve K\big)/(T(\breve\BQ_p)\cap \breve I) .$$
We set
\[
  \Adm(\{\mu\})^K=W_K \Adm(\{\mu\}) W_K, \quad
    \Adm(\{\mu\})_K=W_K \backslash \Adm(\{\mu\})^K /W_K.
\]

Let ${}^K \tW \subset \tW$ be the set of minimal elements in their right cosets. It is proved in \cite[Theorem 6.1]{He-KR} (see also \cite[Proposition 5.1]{HH} for a different proof) that

\begin{theorem}\label{comp}
For any standard parahoric subgroup $K$,
\begin{flalign*}\phantom{\qed} & & \Adm(\{\mu\})^K \cap {}^K \tilde W=\Adm(\{\mu\}) \cap {}^K \tilde W.& &\end{flalign*}
\end{theorem}

\subsection{} Note that $\s(\brK)=\brK$. Let $\brK_\s \subset \brK \times \brK$ be the graph of the Frobenius map $\s$. By \cite[\S 6.1]{HR}, we have a map $$\upsilon_K :  Sh_K  \to G(\breve\BQ_p)/\breve K_\sigma (\breve K_1 \times \breve K_1) \cong {}^K \tW.$$ By \cite[Corollary 6.13]{HR}, the image of $\upsilon_K$ is $\Adm(\{\mu\}) \cap {}^K \tilde W$. For any $x \in \Adm(\{\mu\}) \cap {}^K \tilde W$, we denote by $EKOR_{K, x} = \upsilon_K ^{-1}(x) \subset   Sh_K$ the  {\it Ekedahl-Kottwitz-Oort-Rapoport stratum} (EKOR stratum) of $Sh_K$ attached to $x$.

Now we introduce a partial order on ${}^K \tilde W$.
Let $x, x' \in {}^K \tilde W$, we write $x' \preceq_{K, \sigma} x$ if there exists $w\in W_K$ such that $w x' \sigma(w)^{-1} \le x$. By \cite[4.7]{He-07}, this defines a partial order on ${}^K \tilde W$. It is proved in \cite[Theorem 6.15]{HR} that

\begin{theorem}\label{poset}
For any $x \in \Adm(\{\mu\}) \cap {}^K \tilde W$, $$\overline{EKOR_{K, x}}=\sqcup_{x' \in {}^K \tilde W, x' \preceq_{K, \sigma} x} EKOR_{K, x'}.
$$
\end{theorem}

\section{The $\mu$-ordinary locus and the maximal EKOR strata}

We first describe the maximal EKOR strata.


\begin{proposition}\label{ma-EKOR}
The maximal EKOR strata of $Sh_K$ with respect to the closure relations are $EKOR_{K, t^{\mu'}}$, where $\mu'$ runs over elements in the $W_0$-orbit of $\mu$ with $t^{\mu'} \in {}^K \tW$.
\end{proposition}

\begin{proof}
Let $x \in {}^K \tW$ and $w \in W_K$. Then $$\ell(w x \s(w) \i) \ge \ell(w x)-\ell(w)=\ell(w)+\ell(x)-\ell(w)=\ell(x).$$ Therefore if $x' \preceq_{K, \s} x$, then $\ell(x') \le \ell(x)$. We denote by $W_0(\mu)$ the $W_0$-orbit of $\mu$. Thus if $\mu' \in W_0(\mu)$ with $t^{\mu'} \in {}^K \tW$, then $t^{\mu'}$ is a maximal length element in the index set $\Adm(\{\mu\}) \cap {}^K \tilde W$ of the EKOR strata and hence $EKOR_{K, t^{\mu'}}$ is a maximal EKOR stratum.

On the other hand, suppose that $x \in \Adm(\{\mu\}) \cap {}^K \tilde W$ is a maximal element with respect to the partial order $\preceq_{K, \s}$. By definition $x \le t^{\mu'}$ for some $\mu' \in W_0 (\mu)$. Then $x \le y$, where $y \in W_K t^{\mu'} \cap {}^K \tW$. Noticing that $y \in X_*(T)_{\Gamma_0} \rtimes W_K$ we can write $y=t^\l u$ for some $\l \in X_*(T)_{\Gamma_0}$ and $u \in W_K$. As $y \in {}^K \tW$, $y \i$ sends each simple root in $K$ to a positive affine root. Therefore $t^{-\l}$ sends each simple root in $K$ to a positive affine root. In other words, $t^\l \in {}^K \tW$. As $u \in W_K$, we have $\ell(t^\l u)=\ell(t^\l)-\ell(u)$. Hence $y \le t^\l$. Since $W_K$ is finite and $t^\l \in W_K t^{\mu'} W_K$, we have $\l \in W_0(\mu')=W_0(\mu)$. By the maximality of $x$ we have $x=t^\l$ as desired.
\end{proof}

\subsection{} We recall the relation between the $\s$-conjugacy classes on $\tW$ and on $G(\breve \BQ_p)$.

For any $w \in \tW$, we choose a representative $\dot w \in N(\breve \BQ_p)$. We say that $w$ is $\s$-straight if $\ell(w)=\<\nu(\dot w), 2 \rho\>$, where $\rho$ is the half sum of the positive roots in the reduced root system associated to $\tW$ (see \cite[Proposition 4.21]{PRS}). By \cite[2.4]{H1}, $w$ is $\s$-straight if and only if for any $n \in \BN$, $\ell(w \s(w) \cdots \s^{n-1}(w))=n \ell(w)$. In particular, the definition of $\s$-straight elements is independent of the representatives we choose.

We denote by $B(\tW)_\s$ the set of $\s$-conjugacy classes on $\tW$. We call a $\s$-conjugacy class of $\tW$ {\it straight} if it contains a $\s$-straight element. We denote by $B(\tW)_{\s-str}$ the set of straight $\s$-conjugacy classes of $\tW$. It is proved in \cite[Theorem 3.3]{H1} that

\begin{theorem}\label{s-str}
The map $N(\breve \BQ_p) \to G(\breve \BQ_p)$ induces a bijection $$B(\tW)_{\s-str} \cong B(G).$$
\end{theorem}

\smallskip

We have the following result on the intersection of an $\brI$-double coset with a $\s$-conjugacy class of $B(\breve \BQ_p)$.

\begin{proposition}\label{leq}
Let $b \in G(\breve \BQ_p)$ and $w \in \tW$. Then

(1) If $\brI \dot w \brI \cap [b] \neq \emptyset$, then $\k(w)=\k(b)$ and $\ell(w) \ge \<\nu(b), 2 \rho\>$.

(2) If $\brI \dot w \brI \cap [b] \neq \emptyset$ and $\ell(w)=\<\nu(b), 2 \rho\>$, then $w$ is a $\s$-straight element and $\brI \dot w \brI \subset [b]$.
\end{proposition}

\begin{proof}
Let $\CO \in B(\tW)_{\s-str}$. We write $\CO \preceq_\s w$ if there exists a $\s$-straight element $w' \in \CO$ such that $w' \le w$ in the usual Bruhat order. By \cite[Theorem 2.1 \& 2.7]{He-KR}, $\CO_b \preceq_\s w$, where $\CO_b$ is the straight $\s$-conjugacy class that corresponds to $[b]$. In other words, there exists a $\s$-straight element $w' \in \CO_b$ with $w' \le w$.

By the definition of the Bruhat order, we have $\k(w)=\k(w')=\k(b)$ and $\ell(w) \ge \ell(w')=\<\nu(\dot w), 2 \rho\>=\<\nu(b), 2 \rho\>$. Part (1) is proved.

If moreover, $\ell(w)=\<\nu(b), 2\rho\>$, then $w=w' \in \CO_b$. In other words, $w$ is a $\s$-straight element with $\dot w \in [b]$. By \cite[Theorem 3.7]{H1}, $\brI \dot w \brI \subset [b]$.
\end{proof}

\subsection{} By definition, for any $\s$-conjugacy class $[b]$ in $B(G, \{\mu\})$, $\nu([b]) \le \mu^{\diamond}$. In \cite[Corollary 2.6]{HN}, we give an explicit combinatorial criterion to check whether $\mu^{\diamond}$ equals the Newton point of some $\s$-conjugacy class in $B(G, \{\mu\})$. We will not recall the precise criterion here, but just provide some examples and counterexamples.

\begin{example}\label{eg-1}
Suppose that $G$ is quasi-split. Then $\mu^{\diamond}$ equals the Newton point of some $\s$-conjugacy class in $B(G, \{\mu\})$. Moreover, we prove that

(a) For $\mu' \in W_0(\mu)$, $t^{\mu'}$ is $\s$-straight if and only if $\mu' \in W_0^\s(\mu)$.

If $\mu' \in W_0^\s(\mu)$, then $\nu([t^{\mu'}])=\mu^\diamond$ and $\ell(t^{\mu'})=\<\nu(t^{\mu'}), 2 \rho\>$. So $t^{\mu'}$ is $\s$-straight.

On the other hand, for any $\mu' \in W_0(\mu)$, by \cite{He-KR} we have $\nu(t^{\mu'}) \preceq \mu^\diamond$. If $t^{\mu'}$ is $\s$-straight, then $$\<\nu(t^{\mu'}), 2 \rho\>=\ell(t^{\mu'})=\ell(t^\mu)=\<\mu^\diamond, 2 \rho\>.$$ Thus $\nu(t^{\mu'})=\mu^\diamond=\nu(t^\mu)$. By Proposition \ref{leq} (2), $t^{\mu'}$ is $\s$-conjugate to $t^\mu$. In other words, there exists $z \in W_0$ and $\l \in \mu+(1-\s) X_*(T)_{\Gamma_0}$ such that $z t^{\mu'} \s(z) \i=t^\l$. Thus $z=\s(z)$ and $t^{z(\mu')}=t^\l$. However, $\mu-z(\mu') \in \sum \BN \a_i^\vee$ and $(1-\s) X_*(T)_{\Gamma_0} \cap \sum \BN \a_i^\vee=\{0\}$. Therefore $z(\mu')=\mu$ and $\mu' \in W_0^\s(\mu)$. So (a) is proved.
\end{example}

\begin{example}\label{eg-11}
Assume $G$ is an inner form of $GL_n$ and $\mu$ is minuscule. Then $\tW=\BZ^n \rtimes S_n$. We determine when $\mu^\diamond$ is the maximal element in $B(G, \{\mu\})$. Suppose $\mu=\o_m^\vee$ with $0 \le m \le n-1$ and $\s=\Ad(\t_1^k)$ for some $0 \le k \le n-1$, where $\t_1$ is the unique length zero element in $t^{\o_1^\vee} W_0$. Then, by \cite[Corollary 2.6]{HN}, $\mu^\diamond$ is the maximal element in $B(G, \{\mu\})$ if and only if $\<\o_k^\vee, \o_m\> \in \BZ$, that is, $m k$ is divisible by $n$.
\end{example}

In the rest of this section, we assume that there exists $[b] \in B(G, \{\mu\})$ with $\nu([b])=\mu^{\diamond}$. We denote by $Sh_K^{\mu-\ord}=\mathit S_{K, [b]}$ and call it the {\it $\mu$-ordinary locus} of $Sh_K$. By \cite[Axiom 3]{HR}, it is an open subvariety of $Sh_K$.

\begin{theorem}\label{union}
We have $$Sh_K^{\mu-\ord}=\sqcup_{\mu' \in W_0(\mu);\ t^{\mu'} \in {}^K \tW \text{ and } t^{\mu'} \text{ is $\s$-straight}} EKOR_{K, t^{\mu'}}.$$

In particular, if $G$ is quasi-split, then $$Sh_K^{\mu-\ord}=\sqcup_{\mu' \in W_0^\s(\mu); t^{\mu'} \in {}^K \tW} EKOR_{K, t^{\mu'}}.$$
\end{theorem}

\begin{proof} By \cite[Axiom 4 \& \S6.1]{HR}, we have the following commutative diagram \[\xymatrix{ & & \breve G(\breve\BQ_p)/\breve K_\sigma (\breve K_1 \times \breve K_1) \\  Sh_K  \ar[r]^-{\Upsilon_K} \ar@/^1pc/[urr]^{ \upsilon_K } \ar@/_1pc/[drr]_{ \delta_K } & G(\breve\BQ_p)/\breve K_\sigma \ar[ur]_{p_K} \ar[dr]^{d_K} & \\ & & B(G)}.\]

We have $$Sh_K^{\mu-\ord}=\sqcup_{x \in \Adm(\{\mu\}) \cap {}^K \tW} (Sh_K^{\mu-\ord} \cap EKOR_{K, x}).$$ If $EKOR_{K, x} \cap Sh_K^{\mu-\ord} \neq \emptyset$, then $\brK_\s \cdot \brI \dot x \brI\cap [b] \supset \Upsilon_K(EKOR_{K, x} \cap Sh_K^{\mu-\ord}) \neq \emptyset$. Therefore $\brI \dot x \brI \cap [b] \neq \emptyset$. By Proposition \ref{leq} (1), $\ell(x) \ge \<\mu, 2 \rho\>$. Since $x \in \Adm(\{\mu\})$, $x \le t^{\mu'}$ for some $\mu' \in W_0(\mu)$. Since $\ell(t^{\mu'})=\ell(t^\mu)=\<\mu, 2 \rho\>$, we must have $x=t^{\mu'}$ and $t^{\mu'} \in {}^K \tW$. By Proposition \ref{leq} (2), $t^{\mu'}$ is $\s$-straight.

On the other hand, suppose $\mu' \in W_0(\mu)$ such that $t^{\mu'}$ is a $\s$-straight element and $t^{\mu'} \in {}^K \tW$. By the same argument in Example \ref{eg-1} we have $\nu(t^{\mu'})=\mu^\diamond=\nu([b])$ and hence $[t^{\mu'}]=[b]$. By Proposition \ref{leq} (2), $\brI t^{\mu'} \brI \subset [b]$ and $\brK_\s \cdot \brI t^{\mu'} \brI \subset [b]$. So $\Upsilon_K(EKOR_{K, t^{\mu'}}) \subset \brK_\s \cdot \brI t^{\mu'} \brI \subset [b]$. Thus $EKOR_{K, t^{\mu'}} \subset \d_K \i([b])=Sh_K^{\mu-\ord}$.

The ``in particular'' part follows from Example \ref{eg-1} (a).
\end{proof}

\subsection{} Recall that $\tW=X_*(T)_{\G_0} \rtimes W_0$. Let $p: \tW \to W_0$ be the projection map. Let $W_\mu \subset W_0$ be the stabilizer of $\mu$. We have the following criteria for the density of the $\mu$-ordinary locus.

\begin{theorem}\label{density}
The following conditions are equivalent:

(1) The $\mu$-ordinary locus $Sh_K^{\mu-\ord}$ is dense in $Sh_K$;

(2) the $\mu$-ordinary locus $Sh_K^{\mu-\ord}$ is the union of all maximal EKOR strata;

(3) For any $\mu' \in W_0(\mu)$ with $t^{\mu'} \in {}^K \tW$, $t^{\mu'}$ is a $\s$-straight element in $\tW$.

If moreover, $G$ is quasi-split, then the above conditions are also equivalent to

(4) $p(W_K) W_0^\s W_\mu=W_0$.
\end{theorem}

\begin{remark}\label{rk}
Note that the conditions (3) and (4) are group-theoretic conditions and we will see below that for quasi-split groups, these two conditions are equivalent even if $(G, \mu)$ does not come from a Shimura variety.
\end{remark}

\begin{proof}
(1)$\Rightarrow$(2): Since $Sh_K^{\mu-\ord}$ is dense, $Sh_K^{\mu-\ord}$ intersects all the maximal EKOR strata. By Theorem \ref{union}, $Sh_K^{\mu-\ord}$ is a union of maximal EKOR strata. Therefore it must be the union of all the maximal EKOR strata.

(2)$\Rightarrow$(1): By Theorem \ref{poset}, the union of all maximal EKOR strata is dense in $Sh_K$. As $Sh_K^{\mu-\ord}$ is the union of all maximal EKOR strata, it is dense in $Sh_K$.

The equivalence of (2) and (3) follows from Theorem \ref{union}.

In the rest of the proof, we assume that $G$ is quasi-split.

(3)$\Rightarrow$(4): Let $z \in W_0$. Then there exists $u \in p(W_K)$ such that $t^{uz(\mu)} \in {}^K \tW$. So $t^{uz(\mu)}$ is $\s$-straight. In view of Example \ref{eg-1} we have $uz(\mu) \in W_0^\s(\mu)$, which means $z \in p(W_K) W_0^\s W_\mu$ as desired.

(4)$\Rightarrow$(3): Let $\mu' \in W_0(\mu)$ such that $t^{\mu'} \in {}^K \tW$, that is, $\mu'$ lies in the dominant Weyl chamber $C_{\underline K}$ for the set of positive roots in $K$ (see \cite[\S 5.7]{GHN}). Write $\mu'=u w(\mu)$ with $u \in p(W_K)$ and $w \in W_0^\s$. Via left multiplication by an appropriate element of $p(W_K)$ we can assume furthermore that $u w (C) \subset C_{\underline K}$, where $C$ denotes the dominant Weyl chamber for the set of positive roots in $G$. Thus $$\s(u) w (C)=\s(u) \s(w) (C) \subset \s(C_{\underline K})=C_{\underline K},$$ which means $\s(u) w (C)=u w(C)$ and hence $u=\s(u)$. So $\mu'=u w(\mu) \in W_0^\s(\mu)$. By Example \ref{eg-1}, $t^{\mu'}$ is $\s$-straight.
\end{proof}

In the case that $G$ is residually split, the action of $\s$ on $W_0$ is trivial. In this case, the condition (4) of Theorem \ref{density} is automatically satisfied. Thus

\begin{corollary}
If $G$ is residually split, then the $\mu$-ordinary locus $Sh_K^{\mu-\ord}$ is dense in $Sh_K$.
\end{corollary}

Also if $K$ is special maximal, then $p(W_K)=W_0$. Therefore,

\begin{corollary}
If $G$ is quasi-split and $K$ is a special maximal parahoric subgroup, then the $\mu$-ordinary locus $Sh_K^{\mu-\ord}$ is dense in $Sh_K$.
\end{corollary}

\begin{example} \label{exm}
Consider the case where $G=GU_4$. Here $\s$ induces the nontrivial diagram automorphism on the relative Weyl group $W_0=S_4$. We label the simple reflections in $W_0$ by $s_1, s_2, s_3$ and the unique affine simple reflection that is not in $W_0$ by $s_0$ in the usual way.  By the criterion in Theorem \ref{density}, it is easy to see that for $\mu=\o^\vee_1$ or $\o_3^\vee$, $Sh_K^{\mu-\ord}$ is dense for any parahoric $K$; for $\mu=\o_2^\vee$, $Sh_K^{\mu-\ord}$ is dense if and only if $\{s_1, s_3\} \subset W_K$.

\end{example}

\subsection{} By Theorem \ref{density}, the $\mu$-ordinary locus is dense if and only if

(*) $t^{\mu'}$ is $\s$-straight for any $\mu' \in W_0(\mu)$.

In the rest of this section, we study in details when the condition (*) holds. Note that this is a group-theoretic question and we do not assume in below that $(G, \mu)$ comes from a Shimura variety.

Let $\tW^{\ad}$ be the Iwahori-Weyl group of the adjoint group $G^{\ad}$ of $G$ and $\tW \to \tW^{\ad}$ be the natural map. We denote the automorphism of $\tW^{\ad}$ induced by $\s$ again by $\s$. By definition, $w \in \tW$ is $\s$-straight if and only if its image in $\tW^{\ad}$ is $\s$-straight. Therefore, it suffices to study the condition (*) for adjoint groups.

Now assume that $G$ is adjoint. We may decompose $G$ as $G \cong G_1 \times \cdots \times G_r$, where each $G_i$ is adjoint and simple over $\BQ_p$. We then have the decompositions $\tW \cong \tW_1 \times \cdots \times \tW_r$ and $W_0 \cong W_{0, 1} \times \cdots \times W_{0, r}$, where $\tW_i$ is the Iwahori-Weyl group of $G_i$ and $W_{0, i}$ is the relative Weyl group of $G_i$. We write $\mu$ as $\mu=(\mu_1, \cdots, \mu_r)$. Then it is easy to see that $t^{\mu'}$ is a $\s$-straight element of $\tW$ for any $\mu' \in W_0(\mu)$ if and only if for any $i$, $t^{\mu'_i}$ is a $\s$-straight element of $\tW_i$ for any $\mu'_i \in W_{0, i}(\mu_i)$. Thus, it suffices to study the condition (*) for adjoint, simple groups over $\BQ_p$.

As explained in \cite[\S 2.4]{GHN}, if $G$ is adjoint and simple over $\BQ_p$, then the action of $\s$ on the set of connected components of the associated affine Dynkin diagram is transitive. We first discuss the case where $G$ is quasi-simple over $\breve \BQ_p$. This assumption is equivalent to the assumption that the affine Dynkin diagram of $\tW$ is connected. For any vertex $i$ of the Dynkin diagram, we denote by $\a_i$ and $\o^\vee_i$ the corresponding simple root and fundamental coweight respectively.

\begin{proposition}\label{sh-I}
We use the same labeling of the Dynkin diagram as in \cite{Bour}. Suppose that $G$ is quasi-simple over $\breve \BQ_p$ and $\mu \neq 0$. Then the condition (*) holds if and only if $G$ is quasi-split and $(G, \mu, \s)$ is in one of the following cases:

\begin{itemize}
\item The action of $\s$ on $W_0$ is trivial (and there is no further conditions on $\mu$);

\item The group $G$ is of type $A_n$ with $n$ odd, the action of $\s$ on $W_0$ is of order $2$, and $\mu=m \o_i^\vee$, where $m \in \BZ_{\ge 1}$ and $i=1$ or $n$;

\item The group $G$ is of type $D_n$, the action of $\s$ on $W_0$ is of order $2$, and $\mu=m \o_i^\vee$, where $m \in \BZ_{\ge 1}$ and $i$ is a vertex in the Dynkin diagram not fixed by $\s$.
\end{itemize}
\end{proposition}


\begin{proof}
We first assume $G$ is quasi-split. By Theorem \ref{density} and Remark \ref{rk}, the condition (*) is equivalent to the condition that $W_0^\s W_\mu=W_0$. Note that $W_0^\s$ is again a Weyl group, $W_\mu$ is a standard parabolic subgroup of $W_0$ and $W_0^\s \cap W_\mu$ is a standard parabolic subgroup of $W_0^\s$.

It is easy to see that $W_0^\s W_\mu=W_0$ if and only if \[\tag{a} \frac{\sharp W_0^\s}{\sharp W_0^\s \cap W_\mu}=\frac{\sharp W_0}{\sharp W_\mu}.\] The cardinality of finite Weyl groups are well-known (see, e.g \cite{Bour}). One may check case-by-case that the equality (a) holds exactly in the cases listed in the Proposition.

Now we assume $G$ is not quasi-split. As we discussed above, we may and do assume that $G$ is adjoint. If $\o^\vee_i$ is minuscule, we denote by $\t_i$ the corresponding length-zero element in $t^{\o_i^\vee} W_0$. Then up to equivalences, one of the following cases occurs:
\begin{enumerate}[label=(\roman*)]

\item $\s=\Ad(\t)$ for some nontrivial length zero element $\t$ in $\tW$;

\item $W_0$ is of type $A_n$ with $n$ even and $\s=\Ad(\t_1) \circ \s_0$ with $\s_0$ nontrivial;

\item $W_0$ is of type $D_n$ and $\s=\Ad(\t_n) \circ \s_0$ with $\s_0$ a nontrivial involution.
\end{enumerate}

By definition, $t^{\mu'}$ is $\s$-straight if and only if $p(\s)^i(\mu')$ for $i \in \BZ$ are in the closure of a single Weyl chamber, where $p(\s) \in W_0 \rtimes \<\s_0\> \subset GL(X_*(T)_\G \otimes \BR)$ is the linear part of the affine transformation $\s$ on $X_*(T)_\G \otimes \BR$. In particular, if $t^{\mu'}$ is $\s$-straight, then $\<\mu', \a\> \<p(\s)(\mu'), \a\> \ge 0$ for any root $\a$.

In Case (i), we have $p(\s)=w_0 w_0^L$, where $w_0$ is the longest element in $W_0$ and $w_0^L$ is the longest element in a proper standard parabolic subgroup of $W_0$. In particular, $p(\s) \in W_0 - W_{\mu}$ and $p(\s)(\mu) \neq \mu$. Therefore $p(\s)(\mu), \mu \in W_0(\mu)$ are not in the closure of a single Weyl chamber and hence $t^{\mu}$ is not $\s$-straight.

In Case (ii), we have $p(\s)(\a_1)=-\th$ and $p(\s)(\th)=-\a_1$, where $\th$ is the highest root. Suppose that $\mu=(a_1, \cdots, a_n)$ with $a_1+\cdots+a_n=0$. As $\mu$ is dominant and $\mu \neq 0$, we have $a_1 \ge \cdots \ge a_n$ and $a_1 \neq a_n$. If $a_1>a_{n-1}$, then we set $\mu'=(a_1, a_{n-1}, a_2, \cdots, a_{n-2}, a_n)$. So $\<\mu', \th\>, \<\mu', \a_1\> >0$ and hence $\<\mu', \th\> \<p(\s)(\mu'), \th\><0$. If $a_1=a_{n-1}$, then we set $\mu'=(a_n, a_1, \cdots, a_1)$. So $\<\mu', \th\>, \<\mu', \a_1\> <0$ and hence $\<\mu', \th\> \<p(\s)(\mu'), \th\><0$. In either case, $\mu'$ is not $\s$-straight.

In case (iii), we have $p(\s)(\b)=-\b$, where $\b=\a_2+\cdots+\a_n$. Let $w \in W_0$ such that $w(\b)$ is the highest root. Since $\mu \neq 0$, $\<\mu, w(\b)\>>0$. Set $\mu'=w \i(\mu)$. Then $\<\mu', \b\>\<p(\s)(\mu'), \b\> < 0$ and $t^{\mu'}$ is not $\s$-straight.
\end{proof}

\subsection{} Now we discuss the case where $G$ is adjoint and simple over $\BQ_p$. As we mentioned before, the action of $\s$ on the set of connected components of the associated affine Dynkin diagram is transitive. We may write $\tW$ as $\tW=\tW_1 \times \cdots \times \tW_m$, where $\tW_1 \cong \cdots \cong \tW_m$ with connected affine Dynkin diagram and $\s(\tW_1)=\tW_2, \cdots, \s(\tW_m)=\tW_1$. Let $\mu'=(\mu'_1, \cdots, \mu'_m)$. If $t^{\mu'}$ is a $\s$-straight element in $\tW$, then $t^{\sum_{i=1}^m \s^{m-i}(\mu'_i)}$ is a $\s^m$-straight element of $\tW_m$. By Proposition \ref{sh-I}, if the condition (*) holds for $(G, \mu)$, then $\s^m$ stabilizes the relative Weyl group $W_{0, m}$ of $\tW_m$. Then after choosing a suitable special vertex, we may assume that $\s$ stabilizes the relative Weyl group of $\tW$. Now by Theorem \ref{density} (4), the condition (*) holds if and only if $W_0^\s W_\mu=W_0$.

If there exist $i \neq j$ such that $\mu_i \neq 0$ and $\mu_j \neq 0$. Then the condition $W_0^\s W_\mu=W_0$ implies that $W_1 W_2=W_{0, j}$, where $W_1$ is the stabilizer of $\s^{j-i}(\mu_i)$ in $W_{0, j}$ and $W_2$ is the stabilizer of $\mu_j$ in $W_{0, j}$. Since $\mu_i, \mu_j \neq 0$, $W_1$ and $W_2$ are proper standard parabolic subgroups of $W_{0, j}$. It is easy to check that $W_1 W_2 \subsetneqq W_{0, j}$. Thus $W_0^\s W_\mu \neq W_0$.

If there is only one $i$ such that $\mu_i \neq 0$, then the condition $W_0^\s W_\mu=W_0$ is equivalent to the condition $W_{0, i}^{\s^m} W_{i, \mu_i}=W_{0, i}$, where $W_{i, \mu_i}$ is the stabilizer of $\mu_i$ in $W_{0, i}$. The latter condition is studied in Proposition \ref{sh-I} and it happens exactly in one of the three listed cases there.

\subsection{} Now we apply the above criterion to the symplectic Shimura varieties with Iwahori level structure. Let $G=\text{Res}_{E/\BQ_p} Sp_{2n}$, where $E$ is a finite extension of $\BQ_p$. Then the Iwahori-Weyl group of $G$ is $\tW=\tW_1 \times \cdots \times \tW_m$, where $\tW_1 \cong \tW_m$ is the affine Weyl group of type $C_n$ and $m$ is the residue degree of the extension $E/\BQ_p$. The action of $\s$ permutes $\tW_1, \cdots, \tW_m$ and $\s^m$ is the identity map on $\tW$. The dominant coweight $\mu$ can be written as $(r \o^\vee_n, \cdots, r \o^\vee_n)$, where $r$ is the ramification index of the extension $E/\BQ_p$. By the above discussion, the $\mu$-ordinary locus $Sh_I^{\mu-\ord}$ is dense if and only if $m=1$, i.e., the extension $E/\BQ_p$ is totally ramified. This result is first obtained by Hartwig in \cite[Theorem 1.1.1]{Ha}.

\section{Maximal Newton stratum}\label{non-qs}

\subsection{} In this section, we assume that there is no $\s$-conjugacy class $[b] \in B(G, \{\mu\})$ such that $\nu([b])=\mu^\diamond$. The main result of \cite{HN} shows that there still exists a unique maximal element $[b]_{\max}$ in $B(G, \{\mu\})$. However, the explicit description of such maximal element is much more complicated. We give an example below.

\begin{example} \label{ex}
Consider the case where $G=D^\times$, where $D$ is a central division algebra of degree $8$. In this case $\tW=\BZ^8 \rtimes S_8$. Let $\t_1$ be the unique length zero element in $t^{\o_1^\vee} W_0$. Then the induced action of $\s$ on $\tW$ equals $\Ad(\t_1^i)$ for some $i \in \BZ$ coprime to $8$. We adopt the usual labeling of the Dynkin diagram of $S_8$ by $1 \le i \le 7$, and denote by $\a_i$ and $\o_i^\vee$ the corresponding simple roots and fundamental coweight respectively.

Let $\mu=\o_5^\vee$ and let $[b]_{\max}$ be the unique maximal element in  $B(G, \{\mu\})$. By the discussion in Example \ref{eg-11}, $\nu([b]_{\max})<\mu$. We may compute $\nu([b]_{\max})$ explicitly. If $\s=\Ad(\t_1)$, then $\nu([b]_{\max})=\frac{2}{3} \o_5^\vee$, and one of the $\s$-straight representatives in $[b]_{\max} \cap \Adm(\mu)$ is $t^{s_{\a_1+\cdots+\a_5}(\mu)} s_{\a_6} \cdots s_{\a_2}$. If $\s=\Ad(\t_1^3)$ then $\nu([b]_{\max})=\frac{1}{3} \o_3^\vee + \frac{1}{2} \o_6^\vee$, and one of the $\s$-straight representatives in $[b]_{\max} \cap \Adm(\mu)$ is $t^{s_{\a_1+\cdots+\a_6}(\mu)} s_{\a_4+\a_5} s_{\a_5} s_{s_6+\a_7}s_{\a_3+\a_4+\a_5} $.
\end{example}

\subsection{} We denote by $Sh_K^{\max}=S_{K, [b]_{\max}}$ the maximal Newton stratum of $Sh_K$. By \cite[Axiom 3 and Theorem 5.4]{HR}, $Sh_K^{\max}$ is a nonempty open subset of $Sh_K$. Contrast to the $\mu$-ordinary locus, the behavior of the maximal Newton stratum is very different from what we have seen in section 2. 

\begin{proposition}
Suppose that $\nu([b]_{\max}) \neq \mu^\diamond$. Then the maximal Newton stratum contains some non-maximal EKOR strata.
\end{proposition}

\begin{proof}
By \cite[Theorem 6.17]{He-KR}, there exists a straight element $x$ in $\Adm(\{\mu\}) \cap {}^K \tW$ with $\dot x \in [b]_{\max}$. By the proof of \cite[Theorem 6.18]{He-KR}, $EKOR_{K, x} \subset Sh_K^{\max}$.

By definition, we also have $\<\bar \nu_x, 2 \rho\> =\<\nu([b]_{\max}), 2 \rho\> < \<\mu, 2 \rho\>$. In particular, $x$ is not a maximal length element in $\Adm(\{\mu\}) \cap {}^K \tW$. By Proposition \ref{ma-EKOR}, $EKOR_{K, x}$ is not a maximal EKOR stratum.
\end{proof}

\subsection{}
Another difference is that in general the maximal Newton stratum is not a union of EKOR strata.

\begin{example}
We continue with Example \ref{ex}. Since $Sh_K^{\max}$ is open, there exists some $\mu' \in W_0(\mu)$ such that $Sh_K^{\max} \cap EKOR_{t^{\mu'}} \neq \emptyset$. However, one computes directly that $\nu(t^{\mu'})=0$. This means $EKOR_{t^{\mu'}} \not \subset Sh_K^{\max}$.
\end{example}

\subsection{} We fix an element $b \in [b]_{\max}$ and we  consider the affine Deligne-Lusztig variety $$X(\mu, b)_K=\{g \in G(\breve \BQ_p)/\breve K; g \i b \s(g) \in\cup_{w \in \Adm(\{\mu\})_K} \breve K \dot w \breve K\}.$$ It is the $\overline{\BF}_p$-valued points of a perfect scheme in the sense of Zhu \cite{Zhu} and Bhatt-Scholze \cite{BS}, a locally closed perfect subscheme of the $p$-adic partial flag variety. As explained in \cite[\S 6.2]{GHN}, the map $\Upsilon_K$ maps the maximal Newton stratum to $\BJ_b \backslash X(\mu, b)_K$, where $\BJ_b$ is the $\s$-centralizer of $b$ in $G(\breve \BQ_p)$.

We also consider the group-theoretic analogy of Newton strata and EKOR strata as in \cite[\S 2.12]{He-CDM}.

Let $M_K=\cup_{w \in \Adm(\{\mu\})_K} \breve K \dot w \breve K$ be the analogy of $Sh_K$. For any $x \in \Adm(\{\mu\})^K \cap {}^K \tW$, let $\breve K_\s \cdot \breve I \dot x \breve I$ be the analogy of the EKOR stratum $EKOR_{K, x}$. For any $[b] \in B(G, \{\mu\})$, let $[b] \cap M_K$ be the analogy of Newton stratum $\mathit S_{K,[b]}$. By \cite[\S 2.5]{He-CDM}, all the subsets of $G(\breve \BQ_p)$ here are admissible in the sense of \cite[\S 2.10]{GHKR}. In particular, one may define their dimension $\dim_K$ and their closure in $M_K$.

By \cite[Theorem 2.40]{He-CDM}, $[b]_{\max} \cap M_K$ is open in $M_K$. Thus $\dim_K([b]_{\max} \cap M_K)=\dim_K(M_K)=\<\mu, 2 \rho\>$. Hence by \cite[Theorem 2.33]{He-CDM}, $\dim X(\mu, b)_K=\dim_K(M_K)-\<\nu([b]_{\max}), 2 \rho\>$. Hence

(a) If $\nu([b]_{\max}) \neq \mu^\diamond$, then $\dim X(\mu, b)_K>0$.

\smallskip

Next we discuss the density of the maximal Newton stratum. We give a criterion for the group-theoretic analogous question.

\begin{proposition}\label{gp-den}
The following conditions are equivalent:

(1) The intersection $[b]_{\max} \cap M_K$ is dense in $M_K$;

(2) For any $\mu'$ in the $W_0$-orbit of $\mu$ with $t^{\mu'} \in {}^K \tW$, we have $[b]_{\max} \cap \breve I t^{\mu'} \breve I \neq \emptyset$.
\end{proposition}

\begin{remark}
The condition (2) is equivalent to the condition that $[b]_{\max} \cap \breve K_\s \cdot \breve I t^{\mu'} \breve I \neq \emptyset$ for every $\mu'$. This is the group-theoretic analogy of the condition that the maximal Newton stratum intersects every maximal EKOR stratum.
\end{remark}

\begin{proof}
Similar to the proof of Proposition \ref{ma-EKOR}, the irreducible components of $M_K$ are the closures of $\breve K_\s \cdot \breve I t^{\mu'} \breve I$, where $\mu'$ runs over the elements in the $W_0$-orbit of $\mu$ with $t^{\mu'} \in {}^K \tW$.

(1)$\Rightarrow$(2): Note that $\breve K_\s \cdot \breve I t^{\mu'} \breve I$ is open in $M_K$. Thus if $[b]_{\max} \cap M_K$ is dense in $M_K$, then it intersects every subset $\breve K_\s \cdot \breve I t^{\mu'} \breve I$. As $[b]_{\max}$ is stable under the action of $\breve K_\s$, we have that $[b]_{\max} \cap \breve I t^{\mu'} \breve I \neq \emptyset$.

(2)$\Rightarrow$(1): Let $\mu' \in W_0(\mu)$ with $t^{\mu'} \in {}^K \tW$. Then $\breve K_\s \cdot \breve I \dot t^{\mu'} \breve I$ is irreducible. By \cite[Theorem 3.29]{He-CDM}, $M_K=\sqcup_{[b'] \in B(G, \{\mu\})} ([b'] \cap M_K)$. Therefore there exists $[b']$ such that $[b'] \cap \breve K_\s \cdot \breve I \dot t^{\mu'} \breve I$ is dense in $\breve K_\s \cdot \breve I \dot t^{\mu'} \breve I$. By assumption, $[b]_{\max} \cap \breve K_\s \cdot \breve I \dot t^{\mu'} \breve I \neq \emptyset$. Therefore the closure of $[b']$ contains an element in $[b]_{\max}$. As $[b]_{\max}$ is the unique maximal element in $B(G, \{\mu\})$, by \cite[Theorem 2.40]{He-CDM}, we have $[b']=[b]_{\max}$ and thus $[b]_{\max} \cap \breve K_\s \cdot \breve I \dot t^{\mu'} \breve I$ is dense in $\breve K_\s \cdot \breve I \dot t^{\mu'} \breve I$. Thus the intersection of $[b]_{\max}$ with any irreducible component $Y$ of $M_K$ is dense in $Y$. Hence $[b]_{\max} \cap M_K$ is dense in $M_K$.
\end{proof}

\subsection{} We conjecture that the maximal Newton stratum in $Sh_K$ is dense if and only if it intersects every maximal EKOR stratum. Proposition \ref{gp-den} provides an evidence to this conjecture. It seems that the conjecture would follow from Proposition \ref{gp-den} and some nice properties on the map $\Upsilon_K: Sh_K \to G(\breve \BQ_p)/\breve K_\s$. 

\end{document}